\documentclass[12pt]{amsart}
\usepackage[T1]{fontenc}
\usepackage{geometry}
\usepackage{color}
\usepackage{multirow}
\usepackage{amstext}
\usepackage{amsthm}
\usepackage{amssymb}
\usepackage{stmaryrd}
\usepackage{stackrel}
\usepackage{graphicx}
\usepackage{wasysym}
\usepackage{relsize}

\makeatletter

\providecommand{\tabularnewline}{\\}

\numberwithin{equation}{section}
\numberwithin{figure}{section}
\numberwithin{table}{section}
\theoremstyle{definition}
\newtheorem{thm}{\protect\theoremname}[section]
\theoremstyle{definition}
\newtheorem{prop}[thm]{\protect\propositionname}
\theoremstyle{definition}

\theoremstyle{definition}
\newtheorem{cor}[thm]{\protect\corollaryname}
\theoremstyle{definition}
\newtheorem{example}[thm]{\protect\examplename}
\theoremstyle{remark}
\newtheorem{rem}[thm]{\protect\remarkname}

\usepackage{graphicx,cases,mathtools}

\usepackage{enumitem}
\setlist{itemsep=0pt,topsep=0pt,parsep=1pt,partopsep=0pt}

\usepackage{tikz}
\usetikzlibrary{arrows,shapes,positioning,backgrounds,calc}
\usetikzlibrary{arrows}
\usetikzlibrary{shapes} 
\usetikzlibrary{patterns}
\tikzstyle{vert} = 
[circle, draw, inner sep=0pt, minimum size=3.6mm,text=cyan]
\tikzstyle{select} = 
[circle, draw, line width=.18mm,fill=black,inner sep=0pt, minimum size=1.5mm,text=cyan]
\tikzstyle{unselect} = 
[circle, draw, line width=.18mm,,fill=white, inner sep=0pt, minimum size=1.5mm,text=cyan]
\tikzstyle{small vert} = 
[circle, draw,fill=grey!40, inner sep=0pt, minimum size=3.75mm]
\tikzstyle{tiny vert} = 
[circle, draw,fill=grey!40, inner sep=1.6pt, minimum size=1mm]
\tikzstyle{rect vert} = 
[rectangle, rounded corners, draw, fill=grey!40,inner sep=2.7pt, minimum size=4mm]
\tikzstyle{b} = [draw, very thick, black,-]
\tikzstyle{d} = [draw, black,-stealth]
\tikzstyle{a} = [draw, black,-stealth]
\definecolor{grey}{rgb}{.7, .7, .7}

\definecolor{orange}{RGB}{255,102,0}
\definecolor{ggreen}{RGB}{0,153,0}
\definecolor{darkblue}{RGB}{0,0,255}
\definecolor{purple}{RGB}{153,51,255}
\definecolor{turq}{RGB}{72,209,204}
\definecolor{gray}{RGB}{220,220,220}
\definecolor{orange2}{RGB}{255,100,0}
\definecolor{purple2}{RGB}{159,51,250}
\definecolor{rred}{rgb}{0.9, 0.17, 0.31}
\definecolor{naugreen}{cmyk}{.43,0,.34,.38}
\definecolor{naublue}{cmyk}{1,.72,0,.32}
\definecolor{mediterranean}{cmyk}{.67,0,.08,.3}
\definecolor{rose}{cmyk}{0,1.00,.20,0}
\definecolor{darkorchid}{cmyk}{.6,.9,0,.05}
\definecolor{butterfly}{cmyk}{.95,.59,0,.10}
\definecolor{springgreen}{cmyk}{1.00,0,.70,.02}
\definecolor{darkred}{cmyk}{0,1,1,.5}
\definecolor{nectarine}{cmyk}{0,0.70,1.00,0}
\definecolor{icyblue}{cmyk}{.84,.25,0,.06}
\definecolor{manatee}{rgb}{0.59, 0.6, 0.67}

\usepackage{tikz-3dplot}

\renewcommand*\env@cases[1][1]{%
  \let\@ifnextchar\new@ifnextchar
  \left\lbrace
  \def\arraystretch{#1}%
  \array{@{}l@{\quad}l@{}}%
}

\newcommand{\DD}{\mathsf{D}}

\newcommand{\GG}{\mathsf{G}}
\newcommand{\HH}{\mathsf{H}}
\newcommand{\delay}{\mbox{\larger[-.75]$\sqcap$}}

\makeatother

\providecommand{\corollaryname}{Corollary}
\providecommand{\examplename}{Example}
\providecommand{\propositionname}{Proposition}
\providecommand{\questionname}{Question}
\providecommand{\remarkname}{Remark}
\providecommand{\theoremname}{Theorem}

\begin{document}
\global\long\def\GEN{\text{GEN}}%
\global\long\def\DNG{\text{DNG}}%
\global\long\def\TER{\text{TER}}%
\global\long\def\DNT{\text{DNT}}%
\global\long\def\Dzero{\mathcal{D}_{G,0}}%
\global\long\def\calA{\mathcal{A}}%
\global\long\def\Conv{\text{Conv}}%
\global\long\def\Ex{\text{Ex}}%
\global\long\def\Int{\text{Int}}%
\global\long\def\Opt{\text{Opt}}%
\global\long\def\mex{\text{mex}}%
\global\long\def\nim{\text{nim}}%
\global\long\def\pty{\text{pty}}%
\global\long\def\min{\text{min}}%
\global\long\def\nOpt{\text{nOpt}}%
\global\long\def\type{\text{type}}%

\newcommand{\boxprod}{\mathop{\textstyle\mathsmaller{\square}}}

\title{IMPARTIAL GEODETIC REMOVING GAMES ON GRID GRAPHS}

\author{Bret J.~Benesh}
\address{
Department of Mathematics,
College of Saint Benedict and Saint John's University,
37 College Avenue South,
Saint Joseph, MN 56374-5011, USA
}
\email{bbenesh@csbsju.edu}

\author{Dana C.~Ernst}
\address{
Department of Mathematics and Statistics,
Northern Arizona University PO Box 5717,
Flagstaff, AZ 86011-5717, USA
}
\email{Dana.Ernst@nau.edu, Nandor.Sieben@nau.edu}

\author{Marie Meyer}
\address{
Department of Engineering, Computing, and Mathematical Sciences,
Lewis University,
1 University Pkwy,
Romeoville, IL 60446, USA
}
\email{mmeyer2@lewisu.edu}

\author{Sarah K.~Salmon}
\address{Department of Mathematics,
University of Colorado Boulder Campus Box 395,
2300 Colorado Avenue, Boulder, CO 80309, USA}
\email{Sarah.Salmon@colorado.edu}

\author{N\'andor Sieben}

\thanks{Date: \the\month/\the\day/\the\year}

\begin{abstract}
A subset of the vertex set of a graph is geodetically convex if it contains every vertex on
any shortest path between two elements of the subset. The convex hull of a set of vertices is the smallest convex set containing the set. We study two games in which two players take turns selecting vertices of a graph. 
The last player to move is the winner. The achievement game ends when the convex hull of the unselected vertices does not contain every vertex in the graph. In the avoidance game, the convex hull of the remaining vertices must contain every vertex. We determine the nim-value of these games for the family of grid graphs. We also provide some results for lattice graphs. Key tools in this analysis are delayed gamegraphs, option preserving maps, and case analysis diagrams.
\end{abstract}

\maketitle

\section{Introduction}

The geodetic closure of a vertex subset $S$ of a finite simple graph is the collection of vertices contained on shortest paths between pairs of elements of $S$. Harary~\cite{HARARY1984323}, followed by Buckley and Harary~\cite{BuckleyHarary86}, introduced two impartial geodetic closure games on graphs based on this concept. In both games, two players take turns selecting previously-unselected vertices, with the geodetic closure of the jointly-selected set being recalculated after each move. In the achievement game $\GEN_{\text{g}}$, the player that causes the geodetic closure to equal the entire vertex set wins. In contrast, in the avoidance game $\DNG_{\text{g}}$, a player loses if they are forced to select a vertex that results in the geodetic closure equaling the full vertex set.  These games were analyzed on several well-known graphs, such as cycles, complete graphs, wheel and generalized wheel graphs, complete bipartite graphs, complete multipartite graphs, hypercubes, the Petersen graph, coronas, complete block graphs, and split graphs~\cite{BuckleyHarary85,BuckleyHarary86,FraenkelHarary89,HaynesHenningTiller,Necascova,Wang17}. 

In~\cite{BeneshGeodetic}, we investigated a variation of the geodetic closure games $\GEN$ and $\DNG$ that uses the convex hull instead of the geodetic closure. A subset of vertices of a finite graph is said to be (geodetically) convex if it contains every vertex along shortest paths connecting vertices in the subset. The convex hull of a subset $S$ of vertices is the smallest convex set containing $S$. The difference between the two operators is that the convex hull iteratively applies the geodetic closure until the set stabilizes. Note that the convex hull is a closure operator while the geodetic closure is not, despite its name. The convex hull and geodetic closure operators are the same for many families of graphs while they are different on others such as hypercube graphs, complete multipartite graphs, and generalized wheel graphs.  While~\cite{BuckleyHarary85,BuckleyHarary86,FraenkelHarary89,HaynesHenningTiller} only determine the outcome of geodetic closure games, we computed the nim-values, and hence the outcomes, of the convex hull games for several families of graphs. Of course, this implies that we simultaneously determined the nim-values for the geodetic closure games in instances where the convex hull games and geodetic closure games are the same.


The geodetic closure and convex hull games described above are building hypergraph games~\cite{hypergraph}. Removing hypergraph games are a complementary flavor of hypergraph games. In removing convex hull hypergraph games, players take turns selecting vertices of a finite graph until the convex hull of the remaining unselected vertices becomes too small. The achievement game ends when the convex hull of the unselected vertices does not equal the vertex set of the graph. In the avoidance game, the convex hull of the remaining vertices must equal the vertex set of the graph after each player makes a move. 

Table~\ref{table:eightGames} summarizes the eight games we can play on a graph $\Gamma$. In this paper, we analyze removing convex hull hypergraph games $\TER$ and $\DNT$, as indicated by the bold font in the table.   Our main results are Propositions~\ref{prop:DNT} and \ref{prop:TER}, where we determine the nim-values of these families of games for lattice graphs.   

Our method of attack involves the introduction of achievement and avoidance removing games played on a $3\times 3$ matrix.  Each removing game played on a grid graph maps to a matrix game under an option preserving map~\cite{Misha,Basic2024,Li}. Since option preserving maps preserve nim-values \cite{Basic2024}, we can study the original removing games played on grid graphs by analyzing the corresponding simpler matrix games instead.

For a comprehensive treatment of the standard theory of impartial games, we suggest that the reader consult~\cite{albert2007lessons,ONAG,SiegelBook}.

\begin{table}[h]
\begin{center}
\begin{tabular}{ |l|cc| } 
\hline
geodetic closure  & achievement & avoidance \\ 
\hline
 building \cite{BuckleyHarary86,HARARY1984323,HaynesHenningTiller}  & $\GEN_{\text{g}}(\Gamma)$ & $\DNG_{\text{g}}(\Gamma)$  \\ 
 removing  & $\TER_{\text{g}}(\Gamma)$ & $\DNT_{\text{g}}(\Gamma)$ \\ 
 \hline
\end{tabular}
\begin{tabular}{ |l|cc| } 
 \hline
\bf  convex hull  & achievement & avoidance \\ 
\hline
 building \cite{BeneshGeodetic}  & $\GEN(\Gamma)$ & $\DNG(\Gamma)$  \\ 
 \bf removing  & {\bf $\TER(\Gamma)$} & {\bf $\DNT(\Gamma)$} \\ 
 \hline
\end{tabular}
\end{center}
\caption{
\label{table:eightGames}
The eight game variations played on a graph $\Gamma$.  }
\end{table}

\section{Preliminaries}

In this section, we recall some terminology and results that we rely on. We denote the \emph{parity} of an integer $r$ by $\pty(r)$.

\subsection{Convex hull}

Let $\Gamma=(V,E)$ be a simple graph with a nonempty vertex set $V$. A \emph{geodesic} of $\Gamma$ is a shortest path between two vertices. A set $K$ of vertices is \emph{geodetically convex}, or simply \emph{convex}, if it contains every vertex along geodesics connecting vertices in $K$. The \emph{convex hull} of a vertex set $S$ is the intersection $[S]:=\bigcap\{K\mid S\subseteq K \text{ and } K  \text{ is convex}\}$ of all convex sets containing $S$. A comprehensive reference about geodetic convexity is~\cite{PelayoBook}.

Let $P_n$ denote the \emph{path graph} with $n$ vertices. Box product graphs of the form $P_{n_1}\Box\cdots\Box P_{n_d}$ with $2\leq n_1\leq \cdots \leq n_d\geq 3$ are called ($d$-dimensional) \emph{lattice graphs}.  
A 2-dimensional lattice graph $P_{m}\Box P_{n}$ is called a \emph{grid graph}. Figure~\ref{fig:embed} depicts the natural embeddings of $P_{3}\Box P_{4}$ and $P_{2}\Box P_{2}\Box P_3$ into $\mathbb{R}^2$ and $\mathbb{R}^3$, respectively.

\begin{figure}[h!]
\begin{tikzpicture}[scale=.3]
\draw (1,1) grid (4,3);
\node at (1,3) [unselect] {};
\node at (2,3) [unselect] {};
\node at (3,3) [unselect] {};
\node at (4,3) [unselect] {};

\node at (1,2) [unselect] {};
\node at (2,2) [unselect] {};
\node at (3,2) [unselect] {};
\node at (4,2) [unselect] {};

\node at (1,1) [unselect] {};
\node at (2,1) [unselect] {};
\node at (3,1) [unselect] {};
\node at (4,1) [unselect] {};
\end{tikzpicture}
\hfil
\tdplotsetmaincoords{70}{30}
\begin{tikzpicture}[
    tdplot_main_coords,
    scale=.5, 
    vertex/.style={circle, draw, inner sep=0pt, minimum size=1.45mm}
  ]
  \foreach \k in {0,1,2} { 
    \foreach \j in {0,1} { 
      \foreach \i in {0, 1} { 
       \node[vertex] (v-\i-\j-\k) at (\i, \j, \k) {}; 
      }
    }
  }
  \foreach \k in {0,1,2} {
    \foreach \j in {0,1} {
      \draw (v-0-\j-\k) -- (v-1-\j-\k);
    }
  }
  \foreach \k in {0,1,2} {
    \foreach \i in {0,1} {
      \draw (v-\i-0-\k) -- (v-\i-1-\k);
    }
  }
  \foreach \j in {0,1} {
    \foreach \i in {0,1} {
      \draw (v-\i-\j-0) -- (v-\i-\j-1);
      \draw (v-\i-\j-1) -- (v-\i-\j-2);
    }
  }
\end{tikzpicture}
\caption{\label{fig:embed}
The graphs $P_{3}\Box P_{4}$ and $P_{2}\Box P_{2}\Box P_3$.
}
\end{figure}

For simplicity we sometimes omit drawing the edges in figures showing grid graphs.

\begin{example}
Figure~\ref{fig:hull} shows three examples of convex hulls $[S]$ of subsets $S\subseteq V$ of the grid graph $P_3\Box P_5$. The elements of $S$ are denoted by $\circ$. The convex vertex sets of $P_m\Box P_n$ are products of a path in $P_m$ and a path in $P_n$. These occupy rectangular regions on the figures. In fact, $[S]$ occupies the smallest rectangular region containing $S$. On the third example $[S]=V$  because $S$ contains a vertex on each edge (top, bottom, left, right) of the grid. This $S$ is minimal with this property.
\end{example}

\begin{figure}[h!]
\begin{tikzpicture}[scale=.3]
\draw[rounded corners,color=black,fill=yellow!0] (.5,.5)--(.5,2.5)--(3.5,2.5)--(3.5,.5)-- cycle;
\node at (0,0) {$\bullet$};
\node at (0,1) {$\bullet$};
\node at (0,2) {$\bullet$};
\node at (1,0) {$\bullet$};
\node at (1,1) {$\circ$};
\node at (1,2) {$\bullet$};
\node at (2,0) {$\bullet$};
\node at (2,1) {$\bullet$};
\node at (2,2) {$\bullet$};
\node at (3,0) {$\bullet$};
\node at (3,1) {$\bullet$};
\node at (3,2) {$\circ$};
\node at (4,0) {$\bullet$};
\node at (4,1) {$\bullet$};
\node at (4,2) {$\bullet$};
\end{tikzpicture}
\hfil 
\begin{tikzpicture}[scale=.3]
\draw[rounded corners,color=black,fill=yellow!0] (-.5,-.5)--(-.5,1.5)--(4.5,1.5)--(4.5,-.5)-- cycle;
\node at (0,0) {$\circ$};
\node at (0,1) {$\bullet$};
\node at (0,2) {$\bullet$};
\node at (1,0) {$\bullet$};
\node at (1,1) {$\circ$};
\node at (1,2) {$\bullet$};
\node at (2,0) {$\circ$};
\node at (2,1) {$\bullet$};
\node at (2,2) {$\bullet$};
\node at (3,0) {$\bullet$};
\node at (3,1) {$\bullet$};
\node at (3,2) {$\bullet$};
\node at (4,0) {$\bullet$};
\node at (4,1) {$\circ$};
\node at (4,2) {$\bullet$};
\end{tikzpicture}
\hfil 
\begin{tikzpicture}[scale=.3]
\draw[rounded corners,color=black,fill=yellow!0] (-.5,-.5)--(-.5,2.5)--(4.5,2.5)--(4.5,-.5)-- cycle;
\node at (0,0) {$\bullet$};
\node at (0,1) {$\bullet$};
\node at (0,2) {$\circ$};
\node at (1,0) {$\bullet$};
\node at (1,1) {$\bullet$};
\node at (1,2) {$\bullet$};
\node at (2,0) {$\circ$};
\node at (2,1) {$\bullet$};
\node at (2,2) {$\bullet$};
\node at (3,0) {$\bullet$};
\node at (3,1) {$\bullet$};
\node at (3,2) {$\bullet$};
\node at (4,0) {$\bullet$};
\node at (4,1) {$\circ$};
\node at (4,2) {$\bullet$};
\end{tikzpicture}
\caption{\label{fig:hull}
Convex hulls of subsets of $P_3\Box P_5$.  The set $S$ consists of the vertices represented by $\circ$ symbols. Each rectangular framed region indicates the convex hull of $S$.}
\end{figure}


What we observed in the previous example generalizes to higher dimensions. The convex hull $[S]$ of a vertex subset $S$ of a $d$-dimensional lattice graph $\Gamma$ is the whole vertex set $V$ if and only if $S$ contains a vertex on every face of the natural embedding of $\Gamma$ in $\mathbb{R}^d$ as a polytope. 

\begin{example}
The convex hull $[S]$ of the set $S$ of vertices denoted by $\circ$ is the full vertex set $V$ for each graph in Figure~\ref{fig:3DExa}. Each $S$ is minimal with this property. Notice that in the second and fourth graphs not every vertex is on a geodesic between two vertices of $S$. 
\end{example}

\begin{figure}[h!]

\tdplotsetmaincoords{70}{30}
\begin{tikzpicture}[
    tdplot_main_coords,
    scale=.5, 
    vertex/.style={circle, draw, inner sep=0pt, minimum size=1.45mm}
  ]
  \foreach \k in {0,1,2} { 
    \foreach \j in {0,1} { 
      \foreach \i in {0, 1} { 
       \node[vertex,fill=black] (v-\i-\j-\k) at (\i, \j, \k) {}; 
      }
    }
  }
  \node[vertex,fill=white] at (0,0,0) {};
  \node[vertex,fill=white] at (1,1,2) {};
  \foreach \k in {0,1,2} {
    \foreach \j in {0,1} {
      \draw (v-0-\j-\k) -- (v-1-\j-\k);
    }
  }
  \foreach \k in {0,1,2} {
    \foreach \i in {0,1} {
      \draw (v-\i-0-\k) -- (v-\i-1-\k);
    }
  }
  \foreach \j in {0,1} {
    \foreach \i in {0,1} {
      \draw (v-\i-\j-0) -- (v-\i-\j-1);
      \draw (v-\i-\j-1) -- (v-\i-\j-2);
    }
  }
\end{tikzpicture}
\hfil
\begin{tikzpicture}[
    tdplot_main_coords,
    scale=.5, 
    vertex/.style={circle, draw, inner sep=0pt, minimum size=1.45mm}
  ]
  \foreach \k in {0,1,2} { 
    \foreach \j in {0,1} { 
      \foreach \i in {0, 1} { 
       \node[vertex,fill=black] (v-\i-\j-\k) at (\i, \j, \k) {}; 
      }
    }
  }
  \node[vertex,fill=white] at (1,0,2) {};
  \node[vertex,fill=white] at (1,1,0) {};
  \node[vertex,fill=white] at (0,1,1) {};
  \foreach \k in {0,1,2} {
    \foreach \j in {0,1} {
      \draw (v-0-\j-\k) -- (v-1-\j-\k);
    }
  }
  \foreach \k in {0,1,2} {
    \foreach \i in {0,1} {
      \draw (v-\i-0-\k) -- (v-\i-1-\k);
    }
  }
  \foreach \j in {0,1} {
    \foreach \i in {0,1} {
      \draw (v-\i-\j-0) -- (v-\i-\j-1);
      \draw (v-\i-\j-1) -- (v-\i-\j-2);
    }
  }
\end{tikzpicture}
\hfil
\begin{tikzpicture}[
    tdplot_main_coords,
    scale=.5, 
    vertex/.style={circle, draw, inner sep=0pt, minimum size=1.45mm}
  ]
  \foreach \k in {0,1,2} { 
    \foreach \j in {0,1} { 
      \foreach \i in {0, 1} { 
       \node[vertex,fill=black] (v-\i-\j-\k) at (\i, \j, \k) {}; 
      }
    }
  }
  \node[vertex,fill=white] at (0,0,2) {};
  \node[vertex,fill=white] at (0,0,0) {};
  \node[vertex,fill=white] at (1,1,1) {};
  \foreach \k in {0,1,2} {
    \foreach \j in {0,1} {
      \draw (v-0-\j-\k) -- (v-1-\j-\k);
    }
  }
  \foreach \k in {0,1,2} {
    \foreach \i in {0,1} {
      \draw (v-\i-0-\k) -- (v-\i-1-\k);
    }
  }
  \foreach \j in {0,1} {
    \foreach \i in {0,1} {
      \draw (v-\i-\j-0) -- (v-\i-\j-1);
      \draw (v-\i-\j-1) -- (v-\i-\j-2);
    }
  }
\end{tikzpicture}
\hfil
\begin{tikzpicture}[
    tdplot_main_coords,
    scale=.5, 
    vertex/.style={circle, draw, inner sep=0pt, minimum size=1.45mm}
  ]
  \foreach \k in {0,1,2} { 
    \foreach \j in {0,1} { 
      \foreach \i in {0, 1} { 
       \node[vertex,fill=black] (v-\i-\j-\k) at (\i, \j, \k) {}; 
      }
    }
  }
  \node[vertex,fill=white] at (0,0,0) {};
  \node[vertex,fill=white] at (0,0,2) {};
  \node[vertex,fill=white] at (0,1,1) {};
  \node[vertex,fill=white] at (1,0,1) {};
  \foreach \k in {0,1,2} {
    \foreach \j in {0,1} {
      \draw (v-0-\j-\k) -- (v-1-\j-\k);
    }
  }
  \foreach \k in {0,1,2} {
    \foreach \i in {0,1} {
      \draw (v-\i-0-\k) -- (v-\i-1-\k);
    }
  }
  \foreach \j in {0,1} {
    \foreach \i in {0,1} {
      \draw (v-\i-\j-0) -- (v-\i-\j-1);
      \draw (v-\i-\j-1) -- (v-\i-\j-2);
    }
  }
\end{tikzpicture}

\caption{\label{fig:3DExa}
Four minimal vertex subsets $S$ of $P_2\Box P_2\Box P_3$ satisfying $[S]=V$.}
\end{figure}

\subsection{Geodetic removing games}

We study two removing hypergraph games on $\Gamma$, where two players take turns selecting previously-unselected vertices  until certain conditions are met. A position in both games is the set $P$ of jointly selected vertices. The achievement game \emph{terminate} $\TER(\Gamma)$ ends as soon as the convex hull of the unchosen vertices is no longer $V$, that is, $[V\setminus P]\ne V$. The avoidance game \emph{do not terminate} $\DNT(\Gamma)$ ends as soon as there is no  choice of vertex to select such that the convex hull of the unchosen vertices is still $V$. In other words, $[V \setminus P]=V$ for each position $P$ in this game. We use normal play convention, so the last player to move wins. Removing and building hypergraph games were studied in~\cite{hypergraph}. The building versions of our geodetic convex hull games were studied in \cite{BeneshGeodetic}. Other versions of these games can be found in \cite{BuckleyHarary86,HaynesHenningTiller}.

\begin{example}
In Figure~\ref{fig:TERplay}, we provide a sample of a non-optimal play of the game $\TER(P_3\Box P_5)$.  The last position is terminal because the full rightmost column is selected, making the convex hull smaller than the vertex set.
\end{example}

\begin{figure}[h!]

\hfil
\begin{tikzpicture}[scale=.3]
\draw[rounded corners,color=black,fill=yellow!0] (-.5,-.5)--(4.5,-.5)--(4.5,2.5)--(-.5,2.5)-- cycle;
\node at (0,0) {$\circ$};
\node at (0,1) {$\circ$};
\node at (0,2) {$\circ$};
\node at (1,0) {$\circ$};
\node at (1,1) {$\circ$};
\node at (1,2) {$\circ$};
\node at (2,0) {$\circ$};
\node at (2,1) {$\circ$};
\node at (2,2) {$\circ$};
\node at (3,0) {$\circ$};
\node at (3,1) {$\circ$};
\node at (3,2) {$\circ$};
\node at (4,0) {$\circ$};
\node at (4,1) {$\circ$};
\node at (4,2) {$\circ$};
\end{tikzpicture}
\hfil
\raisebox{12pt}{$\to$}
\hfil
\begin{tikzpicture}[scale=.3]
\draw[rounded corners,color=black,fill=yellow!0] (-.5,-.5)--(4.5,-.5)--(4.5,2.5)--(-.5,2.5)-- cycle;
\node at (0,0) {$\circ$};
\node at (0,1) {$\circ$};
\node at (0,2) {$\circ$};
\node at (1,0) {$\circ$};
\node at (1,1) {$\circ$};
\node at (1,2) {$\circ$};
\node at (2,0) {$\circ$};
\node at (2,1) {$\circ$};
\node at (2,2) {$\circ$};
\node at (3,0) {$\circ$};
\node at (3,1) {$\circ$};
\node at (3,2) {$\circ$};
\node at (4,0) {$\bullet$};
\node at (4,1) {$\circ$};
\node at (4,2) {$\circ$};
\end{tikzpicture}
\hfil
\raisebox{12pt}{$\to$}
\hfil
\begin{tikzpicture}[scale=.3]
\draw[rounded corners,color=black,fill=yellow!0] (-.5,-.5)--(4.5,-.5)--(4.5,2.5)--(-.5,2.5)-- cycle;
\node at (0,0) {$\circ$};
\node at (0,1) {$\circ$};
\node at (0,2) {$\circ$};
\node at (1,0) {$\circ$};
\node at (1,1) {$\circ$};
\node at (1,2) {$\circ$};
\node at (2,0) {$\circ$};
\node at (2,1) {$\bullet$};
\node at (2,2) {$\circ$};
\node at (3,0) {$\circ$};
\node at (3,1) {$\circ$};
\node at (3,2) {$\circ$};
\node at (4,0) {$\bullet$};
\node at (4,1) {$\circ$};
\node at (4,2) {$\circ$};
\end{tikzpicture}
\hfil
\raisebox{12pt}{$\to$}
\hfil
\begin{tikzpicture}[scale=.3]
\draw[rounded corners,color=black,fill=yellow!0] (-.5,-.5)--(4.5,-.5)--(4.5,2.5)--(-.5,2.5)-- cycle;
\node at (0,0) {$\circ$};
\node at (0,1) {$\circ$};
\node at (0,2) {$\circ$};
\node at (1,0) {$\circ$};
\node at (1,1) {$\circ$};
\node at (1,2) {$\circ$};
\node at (2,0) {$\circ$};
\node at (2,1) {$\bullet$};
\node at (2,2) {$\circ$};
\node at (3,0) {$\circ$};
\node at (3,1) {$\circ$};
\node at (3,2) {$\circ$};
\node at (4,0) {$\bullet$};
\node at (4,1) {$\circ$};
\node at (4,2) {$\bullet$};
\end{tikzpicture}
\hfil
\raisebox{12pt}{$\to$}
\hfil
\begin{tikzpicture}[scale=.3]
\draw[rounded corners,color=black,fill=yellow!0] (-.5,-.5)--(4.5,-.5)--(4.5,2.5)--(-.5,2.5)-- cycle;
\node at (0,0) {$\circ$};
\node at (0,1) {$\circ$};
\node at (0,2) {$\bullet$};
\node at (1,0) {$\circ$};
\node at (1,1) {$\circ$};
\node at (1,2) {$\circ$};
\node at (2,0) {$\circ$};
\node at (2,1) {$\bullet$};
\node at (2,2) {$\circ$};
\node at (3,0) {$\circ$};
\node at (3,1) {$\circ$};
\node at (3,2) {$\circ$};
\node at (4,0) {$\bullet$};
\node at (4,1) {$\circ$};
\node at (4,2) {$\bullet$};
\end{tikzpicture}
\hfil
\raisebox{12pt}{$\to$}
\hfil
\begin{tikzpicture}[scale=.3]
\draw[rounded corners,color=black,fill=yellow!0] (-.5,-.5)--(3.5,-.5)--(3.5,2.5)--(-.5,2.5)-- cycle;
\node at (0,0) {$\circ$};
\node at (0,1) {$\circ$};
\node at (0,2) {$\bullet$};
\node at (1,0) {$\circ$};
\node at (1,1) {$\circ$};
\node at (1,2) {$\circ$};
\node at (2,0) {$\circ$};
\node at (2,1) {$\bullet$};
\node at (2,2) {$\circ$};
\node at (3,0) {$\circ$};
\node at (3,1) {$\circ$};
\node at (3,2) {$\circ$};
\node at (4,0) {$\bullet$};
\node at (4,1) {$\bullet$};
\node at (4,2) {$\bullet$};
\end{tikzpicture}
\hfil

\caption{\label{fig:TERplay}
A play in the game $\TER(P_3\Box P_5)$. Rectangular frames indicate convex hulls of the unselected vertices.}
\end{figure}

\subsection{Gamegraphs}

We often use digraphs to model impartial combinatorial games, where the vertices are positions of the game and arrows represent possible moves between positions. More formally, an \emph{optiongraph} \cite{Misha,Basic2024} is a nonempty set $\DD$ of \emph{positions} together with an \emph{option function} $\Opt_{\DD}:\DD\to 2^\DD$. A \emph{play} from position $p$ to position $q$ is a sequence $p=p_0,p_1,\ldots,p_k=q$ such that $p_{i+1}\in\Opt_\DD(p_i)$ for all $i$. In this case we say that $q$ is a \emph{subposition} of $p$. A play is \emph{cyclic} if a position occurs in it more than once. A finite optiongraph is a (finite) \emph{gamegraph} $\GG$ if it has only acyclic plays and every position is a subposition of a specific position called the \emph{starting position}. Note that the starting position is the unique position that is not an option of any other position.  The \emph{nim-value} of a position is defined recursively as $\nim(p):=\mex(\nim(\Opt(p)))$, where $\mex(A)$ is the smallest nonnegative integer missing from the set $A$. The nim-value of $\GG$ is the nim-value of this unique starting position. Note that the nim-value of a game determines its outcome.

The option function $\Opt_{\GG}$, or simply $\Opt$, encodes the same information as the arrow set $\{(p,q)\mid q\in\Opt(p)\}$ of a digraph whose vertices are the positions. In this digraph the starting position is the unique source vertex and there are walks from a position to its subpositions. 

The advantage of the option function over the arrow set is the following.
A function $\beta:\GG\to\HH$ between gamegraphs is called \emph{option preserving} \cite{Basic2024,Li} if $\Opt_\HH(\beta(p))=\beta(\Opt_\GG(p))$ for all $p\in\GG$.  If $\beta$ is an option preserving gamegraph map, then $\nim(\beta(p))=\nim(p)$ for each $p\in\GG$, as seen in \cite[Corollary~4.18]{Basic2024}. This implies that $\GG$ and $\beta(\GG)$ have the same nim-value and hence outcome under normal play.

The \emph{disjunctive sum} $\GG\Box\HH$ of two gamegraphs $\GG$ and $\HH$ is the gamegraph with positions $\GG\times \HH$ and option function $\Opt_{\GG\Box\HH}(p,q):=(\Opt_\GG(p)\times\{q\})\cup(\{p\}\times\Opt_\HH(q))$.
In terms of digraphs, $\GG\Box \HH$ is the box product of $\GG$ and $\HH$. In each turn a player makes a valid move either in $\GG$ or in $\HH$. The nim-value of $\GG\Box\HH$ can be computed as the \emph{nim-sum} 
\[
\nim(\mathsf{G}\Box\mathsf{H})=\nim(\mathsf{G})\oplus\nim(\mathsf{H}),
\]
which requires binary addition without carry by the Sprague--Grundy Theorem.

\begin{example}
Figure~\ref{fig:exa} shows a representative quotient gamegraph for $\DNT(P_2\Box P_3)$. In this quotient, we identified geometrically congruent positions. Each equivalence class is shown with a representative. It is easy to see that the canonical quotient map from the original gamegraph to this quotient gamegraph is option preserving and hence nim-value preserving.  We have labeled positions with their nim-values. Note that the convex hull of the unselected vertices is the full vertex set in every position.
\end{example}

\begin{figure}[h!]
\begin{tikzpicture}[xscale=1.6,yscale=1]
\node[inner sep=0pt,label=left:{\scriptsize $\color{cyan}{0}$}] (top) at (3,5) {
\begin{tikzpicture}[scale=.3,auto]
\draw (1,1) grid (3,2);
\node (1) at (1,2) [unselect] {};
\node (2) at (2,2) [unselect] {};
\node (3) at (3,2) [unselect] {};
\node (4) at (1,1) [unselect] {};
\node (5) at (2,1) [unselect] {};
\node (5) at (3,1) [unselect] {};
\end{tikzpicture}
};

\node[inner sep=0pt,label=left:{\scriptsize $\color{cyan}{3}$}] (1a) at (2,4) {
\begin{tikzpicture}[scale=.3,auto]
\draw (1,1) grid (3,2);
\node (1) at (1,2) [select] {};
\node (2) at (2,2) [unselect] {};
\node (3) at (3,2) [unselect] {};
\node (4) at (1,1) [unselect] {};
\node (5) at (2,1) [unselect] {};
\node (6) at (3,1) [unselect] {};
\end{tikzpicture}
};

\node[inner sep=0pt,label=left:{\scriptsize $\color{cyan}{1}$}] (1b) at (4,4) {
\begin{tikzpicture}[scale=.3,auto]
\draw (1,1) grid (3,2);
\node (1) at (1,2) [unselect] {};
\node (2) at (2,2) [select] {};
\node (3) at (3,2) [unselect] {};
\node (4) at (1,1) [unselect] {};
\node (5) at (2,1) [unselect] {};
\node (6) at (3,1) [unselect] {};
\end{tikzpicture}
};

\node[inner sep=0pt,label=left:{\scriptsize $\color{cyan}{0}$}] (2a) at (2,3) {
\begin{tikzpicture}[scale=.3,auto]
\draw (1,1) grid (3,2);
\node (1) at (1,2) [select] {};
\node (2) at (2,2) [unselect] {};
\node (3) at (3,2) [unselect] {};
\node (4) at (1,1) [unselect] {};
\node (5) at (2,1) [unselect] {};
\node (6) at (3,1) [select] {};
\end{tikzpicture}
};

\node[inner sep=0pt,label=left:{\scriptsize $\color{cyan}{1}$}] (2b) at (1,3) {
\begin{tikzpicture}[scale=.3,auto]
\draw (1,1) grid (3,2);
\node (1) at (1,2) [select] {};
\node (2) at (2,2) [unselect] {};
\node (3) at (3,2) [select] {};
\node (4) at (1,1) [unselect] {};
\node (5) at (2,1) [unselect] {};
\node (6) at (3,1) [unselect] {};
\end{tikzpicture}
};

\node[inner sep=0pt,label=left:{\scriptsize $\color{cyan}{0}$}] (2c) at (4,3) {
\begin{tikzpicture}[scale=.3,auto]
\draw (1,1) grid (3,2);
\node (1) at (1,2) [select] {};
\node (2) at (2,2) [select] {};
\node (3) at (3,2) [unselect] {};
\node (4) at (1,1) [unselect] {};
\node (5) at (2,1) [unselect] {};
\node (6) at (3,1) [unselect] {};
\end{tikzpicture}
};

\node[inner sep=0pt,label=left:{\scriptsize $\color{cyan}{2}$}] (2d) at (3,3) {
\begin{tikzpicture}[scale=.3,auto]
\draw (1,1) grid (3,2);
\node (1) at (1,2) [select] {};
\node (2) at (2,2) [unselect] {};
\node (3) at (3,2) [unselect] {};
\node (4) at (1,1) [unselect] {};
\node (5) at (2,1) [select] {};
\node (6) at (3,1) [unselect] {};
\end{tikzpicture}
};

\node[inner sep=0pt,label=left:{\scriptsize $\color{cyan}{0}$}] (2e) at (5,3) {
\begin{tikzpicture}[scale=.3,auto]
\draw (1,1) grid (3,2);
\node (1) at (1,2) [unselect] {};
\node (2) at (2,2) [select] {};
\node (3) at (3,2) [unselect] {};
\node (4) at (1,1) [unselect] {};
\node (5) at (2,1) [select] {};
\node (6) at (3,1) [unselect] {};
\end{tikzpicture}
};

\node[inner sep=0pt,label=left:{\scriptsize $\color{cyan}{1}$}] (3a) at (3,2) {
\begin{tikzpicture}[scale=.3,auto]
\draw (1,1) grid (3,2);
\node (1) at (1,2) [select] {};
\node (2) at (2,2) [select] {};
\node (3) at (3,2) [unselect] {};
\node (4) at (1,1) [unselect] {};
\node (5) at (2,1) [unselect] {};
\node (6) at (3,1) [select] {};
\end{tikzpicture}
};

\node[inner sep=0pt,label=left:{\scriptsize $\color{cyan}{0}$}] (3b) at (2,2) {
\begin{tikzpicture}[scale=.3,auto]
\draw (1,1) grid (3,2);
\node (1) at (1,2) [select] {};
\node (2) at (2,2) [unselect] {};
\node (3) at (3,2) [select] {};
\node (4) at (1,1) [unselect] {};
\node (5) at (2,1) [select] {};
\node (6) at (3,1) [unselect] {};
\end{tikzpicture}
};

\node[inner sep=0pt,label=left:{\scriptsize $\color{cyan}{1}$}] (3c) at (4,2) {
\begin{tikzpicture}[scale=.3,auto]
\draw (1,1) grid (3,2);
\node (1) at (1,2) [select] {};
\node (2) at (2,2) [select] {};
\node (3) at (3,2) [unselect] {};
\node (4) at (1,1) [unselect] {};
\node (5) at (2,1) [select] {};
\node (6) at (3,1) [unselect] {};
\end{tikzpicture}
};

\node[inner sep=0pt,label=left:{\scriptsize $\color{cyan}{0}$}] (4a) at (3,1) {
\begin{tikzpicture}[scale=.3,auto]
\draw (1,1) grid (3,2);
\node (1) at (1,2) [select] {};
\node (2) at (2,2) [select] {};
\node (3) at (3,2) [unselect] {};
\node (4) at (1,1) [unselect] {};
\node (5) at (2,1) [select] {};
\node (6) at (3,1) [select] {};
\end{tikzpicture}
};

\path [d, shorten >=1.5pt, shorten <=1.5pt] (top) to (1a);
\path [d, shorten >=1.5pt, shorten <=1.5pt] (top) to (1b);
\path [d, shorten >=1.5pt, shorten <=1.5pt] (1a) to (2a);
\path [d, shorten >=1.5pt, shorten <=1.5pt] (1a) to (2b);
\path [d, shorten >=1.5pt, shorten <=1.5pt] (1a) to (2c);
\path [d, shorten >=1.5pt, shorten <=1.5pt] (1a) to (2d);
\path [d, shorten >=1.5pt, shorten <=1.5pt] (1b) to (2c);
\path [d, shorten >=1.5pt, shorten <=1.5pt] (1b) to (2d);
\path [d, shorten >=1.5pt, shorten <=1.5pt] (1b) to (2e);
\path [d, shorten >=1.5pt, shorten <=1.5pt] (2a) to (3a);
\path [d, shorten >=1.5pt, shorten <=1.5pt] (2b) to (3b);
\path [d, shorten >=1.5pt, shorten <=1.5pt] (2c) to (3a);
\path [d, shorten >=1.5pt, shorten <=1.5pt] (2d) to (3a);
\path [d, shorten >=1.5pt, shorten <=1.5pt] (2c) to (3c);
\path [d, shorten >=1.5pt, shorten <=1.5pt] (2d) to (3b);
\path [d, shorten >=1.5pt, shorten <=1.5pt] (2d) to (3c);
\path [d, shorten >=1.5pt, shorten <=1.5pt] (2e) to (3c);
\path [d, shorten >=1.5pt, shorten <=1.5pt] (3a) to (4a);
\path [d, shorten >=1.5pt, shorten <=1.5pt] (3c) to (4a);
\end{tikzpicture}

\caption{
\label{fig:exa}
Representative quotient gamegraph for $\DNT(P_2\Box P_3)$. Vertices depicted by $\bullet$ are already selected, while vertices depicted by $\circ$ are still unselected.
}
\end{figure}

\section{Delay moves}

Pass moves are a feature of many well-known games.  We introduce the following precise description of games that include such a delay.  Let $\mathsf{G}$ be a gamegraph and $\mathsf{D}_{k}$ be the directed
path with positions $k,k-1,\ldots,0$ and option function satisfying $\Opt(i)=\{i-1\}$ for $i\in\{1,\ldots,k\}$ and $\Opt(0)=\emptyset$. The \emph{delayed gamegraph} $\mathsf{G}{\delay}\mathsf{D}_{k}$ has positions $\mathsf{G}\times \mathsf{D}_{k}$.
In this gamegraph, $(q,r)\in\Opt(p,r)$ if and only if $q\in\Opt(p)$,
while $(p,s)\in\Opt(p,r)$ if and only if $s\in\Opt(r)$ and $p$
is not terminal in $\mathsf{G}$. Roughly speaking, this means that
the delayed gamegraph is $\mathsf{G}\Box\mathsf{D}_{k}$
with arrows connecting terminal positions of $\mathsf{G}$ removed.

\begin{example}
Figure~\ref{fig:PvsD} shows an example of the sum and the delayed gamegraphs. The positions are labeled with their nim-values.
\end{example}

\begin{figure}[h!]
\begin{tabular}{cccc}
\begin{tikzpicture}[xscale=.9, yscale=.7,baseline=-27]
\node[vert] (2)  at (1,2) {$\scriptstyle \mathbf{2}$};
\node[vert] (a)  at (1.5,1) {$\scriptstyle \mathbf{0}$};
\node[vert] (1)  at (0.5,1) {$\scriptstyle \mathbf{1}$};
\node[vert] (0)  at (0,0) {$\scriptstyle \mathbf{0}$};
\path[d] (2) to (a);
\path[d] (2) to (1);
\path[d] (1) to (0);
\end{tikzpicture}
&
\begin{tikzpicture}[xscale=.9, yscale=.7,baseline=-27]
\node[vert] (3)  at (0,3) {$\scriptstyle \mathbf{1}$};
\node[vert] (2)  at (0,2) {$\scriptstyle \mathbf{0}$};
\node[vert] (1)  at (0,1) {$\scriptstyle\mathbf{1}$};
\node[vert] (0)  at (0,0) {$\scriptstyle \mathbf{0}$};
\path[d,magenta] (3) to (2);
\path[d,magenta] (2) to (1);
\path[d,magenta] (1) to (0);
\end{tikzpicture}
& 
\begin{tikzpicture}[xscale=.9, yscale=.7]
\node (02) [vert,rectangle]  at (1,2) {$\scriptstyle \mathbf{3}$};
\node (0a) [vert,rectangle]  at (1.5,1) {$\scriptstyle \mathbf{1}$};
\node (01) [vert,rectangle] at (0.5,1) {$\scriptstyle \mathbf{0}$};
\node (00) [vert,rectangle] at (0,0) {$\scriptstyle \mathbf{1}$};

\node (12) [vert]  at (2.5,1) {$\scriptstyle \mathbf{2}$};
\node (1a) [vert]  at (3.0,0) {$\scriptstyle \mathbf{0}$};
\node (11) [vert] at (2.0,0) {$\scriptstyle \mathbf{1}$};
\node (10) [vert] at (1.5,-1) {$\scriptstyle \mathbf{0}$};
 
\node (22) [vert,rectangle]  at (4,0) {$\scriptstyle \mathbf{3}$};
\node (2a) [vert,rectangle]  at (4.5,-1) {$\scriptstyle \mathbf{1}$};
\node (21) [vert,rectangle] at (3.5,-1) {$\scriptstyle \mathbf{0}$};
\node (20) [vert,rectangle] at (3,-2) {$\scriptstyle \mathbf{1}$};
 
\node (32) [vert]  at (5.5,-1) {$\scriptstyle \mathbf{2}$};
\node (3a) [vert]  at (6,-2) {$\scriptstyle \mathbf{0}$};
\node (31) [vert] at (5,-2) {$\scriptstyle \mathbf{1}$};
\node (30) [vert] at (4.5,-3) {$\scriptstyle \mathbf{0}$};

\path[d,magenta] (02) to (12);
\path[d,magenta] (01) to (11);
\path[d,magenta] (0a) to (1a);
\path[d,magenta] (00) to (10);
\path[d,magenta] (12) to (22);
\path[d,magenta] (11) to (21);
\path[d,magenta] (1a) to (2a);
\path[d,magenta] (10) to (20);
\path[d,magenta] (22) to (32);
\path[d,magenta] (21) to (31);
\path[d,magenta] (2a) to (3a);
\path[d,magenta] (20) to (30);
  
\path[d] (02) to (01);
\path[d] (02) to (0a);
\path[d] (01) to (00);
\path[d] (12) to (11);
\path[d] (12) to (1a);
\path[d] (11) to (10);
\path[d] (22) to (21);
\path[d] (22) to (2a);
\path[d] (21) to (20);
\path[d] (32) to (31);
\path[d] (32) to (3a);
\path[d] (31) to (30);
\end{tikzpicture}
& 
\begin{tikzpicture}[xscale=.9, yscale=.7]
\node (02) [vert,rectangle,rounded corners=2]  at (1,2) {$\scriptstyle \mathbf{1}$};
\node (0a) [vert,rectangle,rounded corners=2]  at (1.5,1) {$\scriptstyle \mathbf{0}$};
\node (01) [vert,rectangle,rounded corners=2] at (0.5,1) {$\scriptstyle \mathbf{2}$};
\node (00) [vert,rectangle,rounded corners=2] at (0,0) {$\scriptstyle \mathbf{0}$};

\node (12) [vert]  at (2.5,1) {$\scriptstyle \mathbf{2}$};
\node (1a) [vert]  at (3.0,0) {$\scriptstyle \mathbf{0}$};
\node (11) [vert] at (2.0,0) {$\scriptstyle \mathbf{1}$};
\node (10) [vert] at (1.5,-1) {$\scriptstyle \mathbf{0}$};
 
\node (22) [vert,rectangle,rounded corners=2]  at (4,0) {$\scriptstyle \mathbf{1}$};
\node (2a) [vert,rectangle,rounded corners=2]  at (4.5,-1) {$\scriptstyle \mathbf{0}$};
\node (21) [vert,rectangle,rounded corners=2] at (3.5,-1) {$\scriptstyle \mathbf{2}$};
\node (20) [vert,rectangle,rounded corners=2] at (3,-2) {$\scriptstyle \mathbf{0}$};
 
\node (32) [vert]  at (5.5,-1) {$\scriptstyle \mathbf{2}$};
\node (3a) [vert]  at (6,-2) {$\scriptstyle \mathbf{0}$};
\node (31) [vert] at (5,-2) {$\scriptstyle \mathbf{1}$};
\node (30) [vert] at (4.5,-3) {$\scriptstyle \mathbf{0}$};

\path[d,magenta,dashed,thick] (02) to (12);
\path[d,magenta,dashed,thick] (01) to (11);
\path[d,magenta,dotted,thick] (12) to (22);
\path[d,magenta,dotted,thick] (11) to (21);
\path[d,magenta,dashed,thick] (22) to (32);
\path[d,magenta,dashed,thick] (21) to (31);
  
\path[d] (02) to (01);
\path[d] (02) to (0a);
\path[d] (01) to (00);
\path[d] (12) to (11);
\path[d] (12) to (1a);
\path[d] (11) to (10);
\path[d] (22) to (21);
\path[d] (22) to (2a);
\path[d] (21) to (20);
\path[d] (32) to (31);
\path[d] (32) to (3a);
\path[d] (31) to (30);
\end{tikzpicture}
\\
$\GG$ & $\DD_3$ & $\GG\Box\mathsf{D}_{3}$ & $\GG{\delay}\DD_{3}$
\end{tabular}

\caption{\label{fig:PvsD}
Sum $\GG\Box\mathsf{D}_{3}$ and delayed $\GG{\delay}\DD_{3}$ gamegraphs.
}
\end{figure}

Play on the delayed gamegraph $\GG{\delay}\mathsf{D}_{k}$ is similar to play on the sum $\GG\Box\mathsf{D}_{k}$ with a change in determining the end of the play. Players first pick either $\GG$ or $\DD_k$ and then make a move in the chosen game. But the game ends as soon as a terminal position is reached in $\GG$ even if $\DD_{k}$ still has legal moves.

According to the Sprague--Grundy Theorem, $\nim(p,r)=\nim(p)\oplus\pty(r)$
in $\mathsf{G}\Box\mathsf{D}_{k}$. We have a similar result for
$\mathsf{G}{\delay}\mathsf{D}_{k}$.

\begin{prop}
If $p\in\GG$, then $\nim(p,2r)=\nim(p,0)=\nim(p)$
and $\nim(p,2r+1)=\nim(p,1)$ in the delayed gamegraph $\mathsf{G}{\delay}\mathsf{D}_{k}$.
\end{prop}

\begin{proof}
To show the first result, we use structural induction on the positions of $\mathsf{G}$. If $p$ is a terminal position of $\mathsf{G}$,
then $(p,2r)$ is a terminal position of $\mathsf{G}{\delay}\mathsf{D_{k}}$,
and so $\nim(p,2r)=0=\nim(p)$. Assume $p$ is a nonterminal position
of $\mathsf{G}$ and $2r$ is a position of $\mathsf{D}$. If $r=0$,
then $\nim(p,2r)=\nim(p)$ since $\mathsf{G}$ and the gamegraph induced
by $V_{\mathsf{G}}\times\{0\}$ are isomorphic. So assume $r>0$.
Then
\[
\begin{aligned}\nim(p,2r) & =\mex(\nim(\Opt(p,2r)))\\
 & =\mex(\nim(\Opt(p)\times\{2r\})\cup\{\nim(p,2r-1)\})\\
 & =\mex(\nim(\Opt(p))\cup\{\nim(p,2r-1)\})
\end{aligned}
\]
by induction. Since $(p,2r-2)\in\Opt(p,2r-1)$ and $\nim(p,2r-2)=\nim(p)$
by induction, we must have $\nim(p,2r-1)\ne\nim(p)$. Thus $\nim(p,2r)=\nim(p)$.

Similar argument proves the second result using $\nim(\Opt(p,2r+1))=\nim(\Opt(p,2r-1))$.
\end{proof}
A consequence of this result is that if the second player has a winning
strategy for $\mathsf{G}$, then the second player also has a winning strategy for the delayed game $\mathsf{G}{\delay}\mathsf{D_{2k}}$. Whenever the first player tries to make a delay move in $\mathsf{D_{2k}}$ (along a dotted
arrow in Figure~\ref{fig:PvsD}), the second player can counter this attempt by also
making a move in $\mathsf{D_{2k}}$ (along a dashed arrow).

\begin{cor}
\label{cor:delay}
If $\mathsf{G}$ is a gamegraph, then $\nim(\mathsf{G}{\delay}\mathsf{D}_{2k})=\nim(\mathsf{G})$
and $\nim(\mathsf{G}{\delay}\mathsf{D}_{2k+1})=\nim(\mathsf{G}{\delay}\mathsf{D}_{1})$.
\end{cor}

\begin{example}
Figure~\ref{fig:quest} shows that $\nim(p)=0$ in $\GG$ does not imply that $\nim(p,1)=1$ in $\GG{\delay}\mathsf{D}_1$ for nonterminal $p$. It is clear that $\nim(p,1)\ne\nim(p)$ for all nonterminal $p$ but an actual formula for $\nim(p,1)$ seems elusive. 
\end{example}

\begin{figure}[h!]
\begin{tabular}{ccc}
\begin{tikzpicture}[xscale=.9, yscale=.7]
\node[vert] (3)  at (0,3) {$\scriptstyle \mathbf{0}$};
\node[vert] (2)  at (0,2) {$\scriptstyle \mathbf{2}$};
\node[vert] (1)  at (0,1) {$\scriptstyle\mathbf{1}$};
\node[vert] (0)  at (0,0) {$\scriptstyle \mathbf{0}$};
\path[d] (3) to (2);
\path[d] (2) to (1);
\path[d] (1) to (0);
\path[d] (3) to[bend right=40] (1);
\path[d] (2) to[bend left=30] (0);
\end{tikzpicture}
\hfill
& \hspace{1in} &
\hfill  
\begin{tikzpicture}[xscale=.9, yscale=.7]
\node[vert] (03) at (0,3.2) {$\scriptstyle \mathbf{0}$};
\node[vert] (02) at (0,2.2) {$\scriptstyle \mathbf{2}$};
\node[vert] (01) at (0,1.2) {$\scriptstyle \mathbf{1}$};
\node[vert] (00) at (0,0.2) {$\scriptstyle \mathbf{0}$};
  
\node[vert,rectangle,rounded corners=2] (13) at (2,3.1) {$\scriptstyle \mathbf{3}$};
\node[vert,rectangle,rounded corners=2] (12) at (2,2.1) {$\scriptstyle \mathbf{1}$};
\node[vert,rectangle,rounded corners=2] (11) at (2,1.1) {$\scriptstyle \mathbf{2}$};
\node[vert,rectangle,rounded corners=2] (10) at (2,0.1) {$\scriptstyle \mathbf{0}$};
  
\node[vert] (23) at (4,3) {$\scriptstyle \mathbf{0}$};
\node[vert] (22) at (4,2) {$\scriptstyle \mathbf{2}$};
\node[vert] (21) at (4,1) {$\scriptstyle \mathbf{1}$};
\node[vert] (20) at (4,0) {$\scriptstyle \mathbf{0}$};

\path[d,magenta] (03) to (13);
\path[d,magenta] (02) to (12);
\path[d,magenta] (01) to (11);
  
\path[d,magenta] (13) to (23);
\path[d,magenta] (12) to (22);
\path[d,magenta] (11) to (21);
  
\path[d] (03) to (02);
\path[d] (02) to (01);
\path[d] (01) to (00);
\path[d] (03) to[bend right=40] (01);
\path[d] (02) to[bend left=30] (00);
\path[d] (13) to (12);
\path[d] (12) to (11);
\path[d] (11) to (10);
\path[d] (13) to[bend right=40] (11);
\path[d] (12) to[bend left=30] (10);
\path[d] (23) to (22);
\path[d] (22) to (21);
\path[d] (21) to (20);
\path[d] (23) to[bend right=40] (21);
\path[d] (22) to[bend left=30] (20);
\end{tikzpicture}\\
$\GG$ & & $\GG{\sqcap}\DD_2$
\end{tabular}
\caption{\label{fig:quest}
The nim-values of the positions of $\GG$ and $\GG{\delay}\mathsf{D}_2$.
}
\end{figure}

\section{Matrix games}

Let $M=\left[\begin{smallmatrix}a & b & c\\
d & e & f\\
g & h & i
\end{smallmatrix}\right]$ be a $3\times3$ matrix with nonnegative entries. Let $S:=\{r_{1},r_{3},c_{1},c_{3}\}$
be the set of \emph{edge sums}, where
\[
r_{1}:=a+b+c,\quad r_{3}:=g+h+i,\quad c_{1}:=a+d+g,\quad c_{3}=c+f+i.
\]
We consider two games in which players decrease one of the positive
entries of the matrix by 1. In the $\DNT(M)$ game, every element
of $S$ must remain positive. The $\TER(M)$ game ends as soon as
one of the elements of $S$ becomes $0$. In both games, the last
player to move wins. 

The reason we are interested in matrix games is their connection to
grid games. Let $\mathsf{G}$ be either $\DNT(P_{m}\Box P_{n})$
or $\TER(P_{m}\Box P_{n})$. We define a map $\alpha$ from the positions of $\mathsf{G}$ to $\mathbb{W}^{3\times3}$. We do this by counting the number of unselected vertices in nine regions of $P_{m}\Box P_{n}$.  These regions are the corners, the edges without corners, and the interior, as shown:
\begin{center}
\scalebox{0.75}{
\begin{tabular}{c|ccc|c}
$a$ &  & $b$ &  & $c$\tabularnewline
\hline 
\multirow{2}{*}{$d$} &  & \multirow{2}{*}{$e$} &  & \multirow{2}{*}{$f$}\tabularnewline
 &  &  &  & \tabularnewline
\hline 
$g$ &  & $h$ &  & $i$\tabularnewline
\end{tabular}
}
\par\end{center}
If $m=2$, then some regions are empty and $d=e=f=0$. 

\begin{example}
A position $P$ in a grid game on $P_{4}\Square P_{5}$ and its corresponding
matrix $\alpha(P)$ are
\[
P:
\begin{tikzpicture}[scale=.3,baseline=10pt]
\draw[color=black] (-.5,2.5)--(4.5,2.5);
\draw[color=black] (-.5,.5)--(4.5,.5);
\draw[color=black] (.5,-.5)--(.5,3.5);
\draw[color=black] (3.5,-.5)--(3.5,3.5);
\node at (0,0) {$\bullet$};
\node at (0,1) {$\circ$};
\node at (0,2) {$\circ$};
\node at (0,3) {$\bullet$};
\node at (1,0) {$\circ$};
\node at (1,1) {$\bullet$};
\node at (1,2) {$\circ$};
\node at (1,3) {$\circ$};
\node at (2,0) {$\bullet$};
\node at (2,1) {$\bullet$};
\node at (2,2) {$\bullet$};
\node at (2,3) {$\circ$};
\node at (3,0) {$\bullet$};
\node at (3,1) {$\circ$};
\node at (3,2) {$\bullet$};
\node at (3,3) {$\circ$};
\node at (4,0) {$\circ$};
\node at (4,1) {$\bullet$};
\node at (4,2) {$\circ$};
\node at (4,3) {$\circ$};
\end{tikzpicture}
,\qquad\alpha(P)=\left[\begin{smallmatrix}0 & 3 & 1\\
2 & 2 & 1\\
0 & 1 & 1
\end{smallmatrix}\right].
\]
\end{example}

The next result is easily seen.

\begin{prop}
\label{prop:alpha}
The map $\alpha$ is option preserving between $\DNT(P_{m}\Box P_{n})\to\DNT(M)$
and $\TER(P_{m}\Box P_{n})\to\TER(M)$, where $M$ is the image
of the starting position.
\end{prop}

Note that the image of the starting position of a game on $P_m \Box P_n$ is 
\[
\left[\begin{smallmatrix}1 & b & 1 \\ d & e & d \\ 1 & b & 1\end{smallmatrix}\right],
\] 
where $b=n-2$, $d=m-2$, and $e=(m-2)(n-2)$.

Since option preserving maps preserve nim-values, we have the following.

\begin{cor}
\label{cor:alpha}If $M$ is the image of the starting position through $\alpha$, then
\[
\nim(\DNT(P_{m}\Box P_{n}))=\nim(\DNT(M)),\quad\nim(\TER(P_{m}\Box P_{n}))=\nim(\TER(M)).
\]
\end{cor}

\begin{prop}
\label{prop:DNTmiddle}For each matrix
\[
\nim(\DNT(\left[\begin{smallmatrix}a & b & c\\
d & e & f\\
g & h & i
\end{smallmatrix}\right]))=\nim(\DNT(\left[\begin{smallmatrix}a & b & c\\
d & 0 & f\\
g & h & i
\end{smallmatrix}\right]))\oplus\pty(e).
\]
\end{prop}

\begin{proof}
Note that the game is not finished until the middle entry is $0$, so the game is isomorphic to a game sum. That is,
\[
\DNT(\left[\begin{smallmatrix}a & b & c\\
d & e & f\\
g & h & i
\end{smallmatrix}\right])\cong\DNT(\left[\begin{smallmatrix}a & b & c\\
d & 0 & f\\
g & h & i
\end{smallmatrix}\right])\Box\mathsf{D}_{e}.
\]
The result now follows from nim addition.
\end{proof}

\begin{prop}
\label{prop:TERmiddle}For each matrix
\[
\nim(\TER(\left[\begin{smallmatrix}a & b & c\\
d & e & f\\
g & h & i
\end{smallmatrix}\right]))=\nim(\TER(\left[\begin{smallmatrix}a & b & c\\
d & \pty(e) & f\\
g & h & i
\end{smallmatrix}\right])).
\]
\end{prop}

\begin{proof}
Note that the game is finished as soon as one of the edge sums becomes 0 even if the middle entry is not zero. This essentially means that  decreasing $e$ is a delay move. Hence
\[
\TER(\left[\begin{smallmatrix}a & b & c\\
d & e & f\\
g & h & i
\end{smallmatrix}\right])\cong\TER(\left[\begin{smallmatrix}a & b & c\\
d & 0 & f\\
g & h & i
\end{smallmatrix}\right])\delay\mathsf{D}_{e}.
\]
The result now follows from Corollary~\ref{cor:delay}.
\end{proof}

\section{Grid games}

We study the grid games $\DNT(P_{m}\Box P_{n})$ and $\TER(P_{m}\Box P_{n})$
through their corresponding matrix games using Corollary~\ref{cor:alpha}.
\begin{prop}
\label{prop:DNT}For each grid graph,
$\nim(\DNT(P_{m}\Box P_{n}))=\pty(mn)$.
\end{prop}

\begin{proof}
By Proposition~\ref{prop:DNTmiddle}, the nim-value is $\nim(\DNT(M))\oplus\pty(mn)$, where 
\[
M=
\left[\begin{smallmatrix}1 & b & 1\\
d & 0 & d\\
1 & b & 1
\end{smallmatrix}\right].
\]
The second player wins $\DNT(M)$ by always making a move at an entry that is centrally symmetric to the move of the first player. Note that the position always remains centrally symmetric after a move of the second player. This guarantees that the second player always has an available move. 
%
%
\end{proof}

\begin{rem}\label{rem:CaseAnalysisDiagramNotation}
The strategy for the second player in the previous
proof can be depicted using the diagram in Figure~\ref{fig:DNTeven}
with the assumptions $x,y\ge1$ and $X,Y\ge2$. The arrows of the diagram
represent moves in the game. The labels on the arrows indicate the
choice of the player using the eight cardinal directions. The game positions
after the first player's move are represented by the $\bullet$ symbol.
Some of the $\bullet$ positions can be visited more than once. This
does not allow for infinite play since $x+y$ is decreasing after
each move. The arrow pointing to a box indicates that one of the positions
in the box is reachable. The starting position of the game is dependent
on the value of $m$. This is indicated by the \textcolor{cyan}{${\color{red}\rightsquigarrow}$}
symbol.

We only consider available moves and resulting positions up to symmetry.
For example, let 
\[
P=\left[\begin{smallmatrix}1 & x & 0\\
0 & 0 & 0\\
0 & x & 1
\end{smallmatrix}\right],\quad Q=\left[\begin{smallmatrix}1 & x & 0\\
y & 0 & y\\
0 & x & 1
\end{smallmatrix}\right], \quad R=\left[\begin{smallmatrix}0 & x & 1\\
y & 0 & y\\
1 & x & 0
\end{smallmatrix}\right].
\]
The moves $\stackrel[\,]{N}{\to}$ and $\stackrel[\,]{S}{\to}$ are
considered the same in position $P$ since a central reflection transforms
one to the other while preserving the matrix. Positions $Q$ and $R$
are also considered the same because of a reflection through a vertical
line.
\end{rem}

The point of the admittedly more complex representation in Figure~\ref{fig:DNTeven}, explained in Remark~\ref{rem:CaseAnalysisDiagramNotation}, is that we are going to need these \emph{case analysis diagrams} to describe more complicated strategies. In these diagrams, the nodes are templates that represent sets of positions. Most templates contain parameters. A position belongs to a set described by a template if the matching parameter values satisfy the requirements on the template parameters. The templates are up to congruence, so a position can be rotated and reflected to match a template. The arrows in the diagram represent moves between sets of positions. A label on an arrow indicates the type of the move. 

\begin{figure}[h!]
\centerline{\setlength{\tabcolsep}{.3em}
\renewcommand{\arraystretch}{1}%
\begin{tabular}{crccccccccc|ccc|}
\cline{12-14} \cline{13-14} \cline{14-14} 
 &  &  & $\bullet$ &  &  &  & $\bullet$ &  &  &  & $\bullet$ &  & \tabularnewline
 &  &  & $W\upharpoonleft\downharpoonright E$ &  &  &  & $W\upharpoonleft\downharpoonright E$ &  &  &  & $W\upharpoonleft\downharpoonright E$ &  & \tabularnewline
 & $m\ge3$ & \textcolor{cyan}{${\color{red}\rightsquigarrow}$} & $\left[\begin{smallmatrix}1 & x & 1\\
y & 0 & y\\
1 & x & 1
\end{smallmatrix}\right]$ & $\stackrel{NE}{\to}$ & $\bullet$ & $\stackrel{SW}{\to}$ & $\left[\begin{smallmatrix}1 & x & 0\\
y & 0 & y\\
0 & x & 1
\end{smallmatrix}\right]$ & $\stackrel{NW}{\to}$ & $\bullet$ & $\stackrel{SE}{\to}$ & $\left[\begin{smallmatrix}0 & X & 0\\
Y & 0 & Y\\
0 & X & 0
\end{smallmatrix}\right]$ &  & \tabularnewline
 &  &  & $W\downarrow\phantom{N}$ &  &  &  & $W\downarrow\phantom{N}$ &  &  &  & $W\downarrow\phantom{N}$ &  & \tabularnewline
 &  &  & $\bullet$ &  &  &  & $\bullet$ &  &  &  & $\bullet$ &  & \tabularnewline
 & $m=2$ & \textcolor{red}{\rotatebox{-45}{$\rightsquigarrow$}} & $E\downarrow\phantom{N}$ &  &  &  & $E\downarrow\phantom{N}$ &  &  &  & $E\downarrow\phantom{N}$ &  & \tabularnewline
 & $\bullet$ & $\stackrel[S]{N}{\leftrightharpoons}$ & $\left[\begin{smallmatrix}1 & x & 1\\
0 & 0 & 0\\
1 & x & 1
\end{smallmatrix}\right]$ & $\stackrel{NE}{\to}$ & $\bullet$ & $\stackrel{SW}{\to}$ & $\left[\begin{smallmatrix}1 & x & 0\\
0 & 0 & 0\\
0 & x & 1
\end{smallmatrix}\right]$ & $\stackrel[S]{N}{\rightleftharpoons}$ & $\bullet$ &  & $\left[\begin{smallmatrix}0 & X & 0\\
1 & 0 & 1\\
0 & X & 0
\end{smallmatrix}\right]$ & $\stackrel[S]{N}{\rightleftharpoons}$ & $\bullet$\tabularnewline
 &  &  & $N\downarrow\phantom{N}$ &  &  &  & $N\downarrow\phantom{N}$ &  &  &  & $N\downarrow\phantom{N}$ &  & \tabularnewline
 &  &  & $\bullet$ &  &  &  & $\bullet$ &  &  &  & $\bullet$ &  & \tabularnewline
 &  &  & $S\downarrow\phantom{N}$ &  &  &  & $S\downarrow\phantom{N}$ &  &  &  & $S\downarrow\phantom{N}$ &  & \tabularnewline
 &  &  & $\left[\begin{smallmatrix}1 & 0 & 1\\
0 & 0 & 0\\
1 & 0 & 1
\end{smallmatrix}\right]$ & $\stackrel{NE}{\to}$ & $\bullet$ & $\stackrel{SW}{\to}$ & $\left[\begin{smallmatrix}1 & 0 & 0\\
0 & 0 & 0\\
0 & 0 & 1
\end{smallmatrix}\right]$ &  &  &  & $\left[\begin{smallmatrix}0 & 1 & 0\\
1 & 0 & 1\\
0 & 1 & 0
\end{smallmatrix}\right]_{\vphantom{\Int_{\Int}}}$ &  & \tabularnewline
\cline{12-14} \cline{13-14} \cline{14-14} 
\end{tabular}}

\caption{\label{fig:DNTeven}Case analysis diagram describing a strategy for the second
player in $\protect\DNT(M)$. We assume $x,y\ge1$ and $X,Y\ge2$.
See Remark~\ref{rem:CaseAnalysisDiagramNotation} for an explanation of the notation.}
\end{figure}

Our next goal is to handle $\TER(P_{m}\Box P_{n})$. First, we
present a result about matrix games to handle the $\pty(mn)=0$ case.
\begin{prop}
\label{prop:TERmatrixEven}If $M=\left[\begin{smallmatrix}1 & b & 1\\
d & 0 & d\\
1 & b & 1
\end{smallmatrix}\right]$, then $\nim(\TER(M))=0$.
\end{prop}

\begin{proof}
The second player wins using the strategy depicted in the case analysis diagram of Figure~\ref{fig:TEReven}
with the assumptions $x,y\ge1$ and $X,Y\ge3$. In this diagram we
omitted the moves of the first player for which the second player
has an immediate winning response that ends the game.
\end{proof}

\begin{figure}[h!]
\centerline{\setlength{\tabcolsep}{.3em}
\renewcommand{\arraystretch}{1}%
\begin{tabular}{crccccccccc|ccc|}
\cline{12-14} \cline{13-14} \cline{14-14} 
 &  &  & $\bullet$ &  &  &  & $\bullet$ &  &  &  & $\bullet$ &  & \tabularnewline
 &  &  & $W\upharpoonleft\downharpoonright E$ &  &  &  & $W\upharpoonleft\downharpoonright E$ &  &  &  & $W\upharpoonleft\downharpoonright E$ &  & \tabularnewline
 & $bd\ge1$ & \textcolor{cyan}{${\color{red}\rightsquigarrow}$} & $\left[\begin{smallmatrix}1 & x & 1\\
y & 0 & y\\
1 & x & 1
\end{smallmatrix}\right]$ & $\stackrel{NE}{\to}$ & $\bullet$ & $\stackrel{SW}{\to}$ & $\left[\begin{smallmatrix}1 & x & 0\\
y & 0 & y\\
0 & x & 1
\end{smallmatrix}\right]$ & $\stackrel{NW}{\to}$ & $\bullet$ & $\stackrel{SE}{\to}$ & $\left[\begin{smallmatrix}0 & X & 0\\
Y & 0 & Y\\
0 & X & 0
\end{smallmatrix}\right]$ &  & \tabularnewline
 &  &  & $W\downarrow\phantom{N}$ &  &  &  & $W\downarrow\phantom{N}$ &  &  &  & $W\downarrow\phantom{N}$ &  & \tabularnewline
 &  &  & $\bullet$ &  &  &  & $\bullet$ &  &  &  & $\bullet$ &  & \tabularnewline
 & $bd=0$ & \textcolor{red}{\rotatebox{-45}{$\rightsquigarrow$}} & $E\downarrow\phantom{N}$ &  &  &  & $E\downarrow\phantom{N}$ &  &  &  & $E\downarrow\phantom{N}$ &  & \tabularnewline
 & $\bullet$ & $\stackrel[S]{N}{\leftrightharpoons}$ & $\left[\begin{smallmatrix}1 & x & 1\\
0 & 0 & 0\\
1 & x & 1
\end{smallmatrix}\right]$ &  &  &  & $\left[\begin{smallmatrix}1 & x & 0\\
1 & 0 & 1\\
0 & x & 1
\end{smallmatrix}\right]$ & $\stackrel[S]{N}{\rightleftharpoons}$ & $\bullet$ &  & $\left[\begin{smallmatrix}0 & X & 0\\
2 & 0 & 2\\
0 & X & 0
\end{smallmatrix}\right]$ & $\stackrel[S]{N}{\rightleftharpoons}$ & $\bullet$\tabularnewline
 &  &  & $N\downarrow\phantom{N}$ &  &  &  & $N\downarrow\phantom{N}$ &  &  &  & $N\downarrow\phantom{N}$ &  & \tabularnewline
 &  &  & $\bullet$ &  &  &  & $\bullet$ &  &  &  & $\bullet$ &  & \tabularnewline
 &  &  & $S\downarrow\phantom{N}$ &  &  &  & $S\downarrow\phantom{N}$ &  &  &  & $S\downarrow\phantom{N}$ &  & \tabularnewline
 & $b,d=0$ & \textcolor{cyan}{${\color{red}\rightsquigarrow}$} & $\left[\begin{smallmatrix}1 & 0 & 1\\
0 & 0 & 0\\
1 & 0 & 1
\end{smallmatrix}\right]$ &  &  &  & $\left[\begin{smallmatrix}1 & 1 & 0\\
1 & 0 & 1\\
0 & 1 & 1
\end{smallmatrix}\right]$ &  &  &  & $\left[\begin{smallmatrix}0 & 2 & 0\\
2 & 0 & 2\\
0 & 2 & 0
\end{smallmatrix}\right]_{\vphantom{\Int_{\Int}}}$ &  & \tabularnewline
\cline{12-14} \cline{13-14} \cline{14-14} 
\end{tabular}}

\caption{\label{fig:TEReven}Case analysis diagram describing a strategy for the second
player in $\protect\TER(M)$. We assume $x,y\ge1$ and $X,Y\ge3$.}
\end{figure}

\begin{rem}
\label{rem:TEReven}The second player's strategy in the previous proof
can be described as: 
\begin{enumerate}
\item Win if possible.
\item Otherwise, select the central reflection of the first player's selection.
\end{enumerate}
\end{rem}

To handle the $\pty(mn)=1$ case, we generalize this \emph{central
reflection strategy} to a broader setting.
\begin{cor}
(CR) If $e$ is even, then $\nim(\TER(M))=0$ for $M=\left[\begin{smallmatrix}a & b & c\\
d & e & d\\
c & b & a
\end{smallmatrix}\right]$.
\end{cor}

The \emph{horizontal reflection strategy} uses a similar idea with
a small twist.
\begin{prop}
(HR) If $e$ is even, \textup{$o$, $\tilde{o}$ are odd, $\tilde{o}\ge3$,
and the edge sums are at least $2$,} then $\nim(\TER(M))=0$ for
$M=\left[\begin{smallmatrix}a & b & c\\
e & o & \tilde{o}\\
a & b & c
\end{smallmatrix}\right]$.
\end{prop}

\begin{proof}
It is easy to see that the second player wins using the following
strategy: 
\begin{enumerate}
\item If the top or bottom edge sum is $1$, then make the winning move
on that edge.
\item If $M_{2,2}$ and $M_{2,3}$ have different parities, then select
the odd entry to make both of them even.
\item Otherwise, select the horizontal reflection of the first player's
selection.
\end{enumerate}
\end{proof}
Note that the $\tilde{o}\ge3$ requirement is necessary, since otherwise
the second player can create the losing position 
\[
\left[\begin{smallmatrix}1 & b & 0\\
e & o & 1\\
1 & b & 0
\end{smallmatrix}\right].
\]

We need two more results to handle common strategies called \emph{middle
games A and B}. Both of these strategies rely on the central reflection
strategy (CR) as an endgame. 
\begin{prop}
(MA) If $b\ge2$\textup{,} then $\nim(\TER(M))=1$ for $M=\left[\begin{smallmatrix}1 & b & 1\\
0 & 1 & 0\\
1 & b & 1
\end{smallmatrix}\right]$.
\end{prop}

\begin{proof}
The second player wins $\TER(M)+*1$ using the strategy depicted in
the case analysis diagram of Figure~\ref{fig:MA} with the assumption $x\ge2$. An arrow indicating
the selection of the central entry is labeled by $C$ in the diagram.
\end{proof}

\begin{figure}[h!]
\centerline{\setlength{\tabcolsep}{.3em}
\renewcommand{\arraystretch}{1.3}

\begin{tabular}{ccccccccccccc}
 &  &  &  & $\bullet$ &  &  &  & $\bullet$ &  &  &  & $\bullet$\tabularnewline
 &  &  & \textcolor{red}{\rotatebox{-45}{$\rightsquigarrow$}} & $N\upharpoonleft\downharpoonright S$ &  &  &  & $N,S\upharpoonleft\downharpoonright S,N$ &  &  &  & $N,S\upharpoonleft\downharpoonright S,N$\tabularnewline
$\left[\begin{smallmatrix}1 & x & 1\\
0 & 0 & 0\\
1 & x & 1
\end{smallmatrix}\right]$ & $\stackrel[*1]{C}{\leftarrow}$ & $\bullet$ & $\stackrel[C]{*1}{\leftarrow}$ & $\left[\begin{smallmatrix}1 & x & 1\\
0 & 1 & 0\\
1 & x & 1
\end{smallmatrix}\right]+*1$ & $\stackrel[NE]{NW}{\to}$ & $\bullet$ & $\stackrel[NW]{NE}{\to}$ & $\left[\begin{smallmatrix}0 & x & 0\\
0 & 1 & 0\\
1 & x & 1
\end{smallmatrix}\right]{\scriptstyle +*1}$ & $\overset{C}{\to}$ & $\bullet$ & $\overset{S}{\to}$ & $\left[\begin{smallmatrix}0 & x & 0\\
0 & 0 & 0\\
1 & x-1 & 1
\end{smallmatrix}\right]{\scriptstyle +*1}$\tabularnewline
$\shortdownarrow$ &  &  &  & $N\downarrow\,\,$ &  &  &  & $N\downarrow S$ &  &  &  & $N\downarrow S$\tabularnewline
(CR) &  &  &  & $\bullet$ &  &  &  & $\bullet$ &  &  &  & $\bullet$\tabularnewline
$\shortuparrow$ &  &  &  & $S\downarrow\,\,$ &  &  &  & $S\downarrow N$ &  &  &  & $S\downarrow N$\tabularnewline
$\left[\begin{smallmatrix}1 & 1 & 1\\
0 & 0 & 0\\
1 & 1 & 1
\end{smallmatrix}\right]$ & $\stackrel[*1]{C}{\leftarrow}$ & $\bullet$ & $\stackrel[C]{*1}{\leftarrow}$ & $\left[\begin{smallmatrix}1 & 1 & 1\\
0 & 1 & 0\\
1 & 1 & 1
\end{smallmatrix}\right]+*1$ & $\stackrel[\,]{NW}{\to}$ & $\bullet$ & $\stackrel[\,]{NE}{\to}$ & $\left[\begin{smallmatrix}0 & 1 & 0\\
0 & 1 & 0\\
1 & 1 & 1
\end{smallmatrix}\right]{\scriptstyle +*1}$ & $\stackrel[S]{C}{\to}$ & $\bullet$ & $\stackrel[C]{S}{\to}$ & $\left[\begin{smallmatrix}0 & 1 & 0\\
0 & 0 & 0\\
1 & 0 & 1
\end{smallmatrix}\right]{\scriptstyle +*1}$\tabularnewline
 &  &  &  &  & $\searrow^{N}$ &  &  &  &  &  &  & \tabularnewline
 &  &  &  &  &  & $\bullet$ & $\stackrel[\,]{*1}{\to}$ & $\left[\begin{smallmatrix}1 & 0 & 1\\
0 & 1 & 0\\
1 & 1 & 1
\end{smallmatrix}\right]$ & $\stackrel[S]{C}{\to}$ & $\bullet$ & $\stackrel[C]{S}{\to}$ & $\left[\begin{smallmatrix}1 & 0 & 1\\
0 & 0 & 0\\
1 & 0 & 1
\end{smallmatrix}\right]$\tabularnewline
\end{tabular}}

\caption{\label{fig:MA}The middle game strategy (MA) for the second player
in $\protect\TER(M)+*1$. We assume $x\ge2$.}
\end{figure}

\begin{prop}
(MB) If $a\ge2$\textup{,} then $\nim(\TER(M))=1$ for $M=\left[\begin{smallmatrix}0 & a & 1\\
1 & 1 & 1\\
1 & a & 0
\end{smallmatrix}\right]$.
\end{prop}

\begin{proof}
The second player wins $\TER(M)+*1$ using the strategy depicted in
the case analysis diagram of Figure~\ref{fig:MB} with the assumption $x\ge2$.
\end{proof}

\begin{figure}[h!]
\centerline{\setlength{\tabcolsep}{.3em}
\renewcommand{\arraystretch}{1.3}

\begin{tabular}{ccccccccccccc}
 &  &  &  & $\bullet$ &  &  &  & $\bullet$ &  &  &  & $\bullet$\tabularnewline
 &  &  & \textcolor{red}{\rotatebox{-45}{$\rightsquigarrow$}} & $N\upharpoonleft\downharpoonright S$ &  &  &  & $N,S\upharpoonleft\downharpoonright S,N$ &  &  &  & $N,S\upharpoonleft\downharpoonright S,N$\tabularnewline
$\left[\begin{smallmatrix}0 & x & 1\\
1 & 0 & 1\\
1 & x & 0
\end{smallmatrix}\right]$ & $\stackrel[*1]{C}{\leftarrow}$ & $\bullet$ & $\stackrel[C]{*1}{\leftarrow}$ & $\left[\begin{smallmatrix}0 & x & 1\\
1 & 1 & 1\\
1 & x & 0
\end{smallmatrix}\right]+*1$ & $\stackrel[NE]{W}{\to}$ & $\bullet$ & $\stackrel[W]{NE}{\to}$ & $\left[\begin{smallmatrix}0 & x & 0\\
0 & 1 & 1\\
1 & x & 0
\end{smallmatrix}\right]+*1$ & $\overset{C}{\to}$ & $\bullet$ & $\overset{S}{\to}$ & $\left[\begin{smallmatrix}0 & x & 0\\
0 & 0 & 1\\
1 & x-1 & 0
\end{smallmatrix}\right]+*1$\tabularnewline
$\shortdownarrow$ &  &  &  & $N\downarrow\,\,$ &  &  &  & $N\downarrow S$ &  &  &  & $N\downarrow S$\tabularnewline
(CR) &  &  &  & $\bullet$ &  &  &  & $\bullet$ &  &  &  & $\bullet$\tabularnewline
$\shortuparrow$ &  &  &  & $S\downarrow\,\,$ &  &  &  & $S\downarrow N$ &  &  &  & $S\downarrow N$\tabularnewline
$\left[\begin{smallmatrix}0 & 1 & 1\\
1 & 0 & 1\\
1 & 1 & 0
\end{smallmatrix}\right]$ & $\stackrel[*1]{C}{\leftarrow}$ & $\bullet$ & $\stackrel[C]{*1}{\leftarrow}$ & $\left[\begin{smallmatrix}0 & 1 & 1\\
1 & 1 & 1\\
1 & 1 & 0
\end{smallmatrix}\right]+*1$ & $\stackrel[NE]{W}{\to}$ & $\bullet$ & $\stackrel[W]{NE}{\to}$ & $\left[\begin{smallmatrix}0 & 1 & 0\\
0 & 1 & 1\\
1 & 1 & 0
\end{smallmatrix}\right]+*1$ & $\stackrel[S]{C}{\to}$ & $\bullet$ & $\stackrel[C]{S}{\to}$ & $\left[\begin{smallmatrix}0 & 1 & 0\\
0 & 0 & 1\\
1 & 0 & 0
\end{smallmatrix}\right]+*1$\tabularnewline
\end{tabular}}

\caption{\label{fig:MB}The middle game strategy (MB) for the second player
in $\protect\TER(M)+*1$. We assume $x\ge2$.}
\end{figure}

Now we provide an opening strategy.
\begin{prop}
\label{prop:TERmatrixOdd}If $b$ and $d$ are odd\textup{,} then
$\nim(\TER(M))=1$ for $M=\left[\begin{smallmatrix}1 & b & 1\\
d & 1 & d\\
1 & b & 1
\end{smallmatrix}\right]$.
\end{prop}

\begin{proof}
The second player wins $\TER(M)+*1$ using the strategy depicted in
the case analysis diagram of Figure~\ref{fig:TERodd} with the assumptions $x,y\ge1$, $X,Y\ge3$,
and $x,y,X,Y$ are odd.
\end{proof}

\begin{figure}[h!]
\centerline{\setlength{\tabcolsep}{.3em}
\renewcommand{\arraystretch}{.5}

\noindent %
\begin{tabular}{ccccc|ccccccc}
\textcolor{cyan}{${\color{red}\rightsquigarrow}$} & $\left[\begin{smallmatrix}1 & x & 1\\
y & 1 & y\\
1 & x & 1
\end{smallmatrix}\right]+*1$ & $\stackrel[*1]{C}{\to}$ & $\bullet$ & \multicolumn{1}{c}{$\stackrel[C]{*1}{\to}$} & $\left[\begin{smallmatrix}1 & x & 1\\
y & 0 & y\\
1 & x & 1
\end{smallmatrix}\right]+*0$ & $\shortrightarrow$ & (CR) &  &  &  & \tabularnewline
 &  &  &  & \multicolumn{1}{c}{} &  &  &  &  &  &  & \tabularnewline
\cline{6-8} \cline{7-8} \cline{8-8} 
 & \vphantom{%
\begin{tabular}{c}
\tabularnewline
\tabularnewline
\tabularnewline
\tabularnewline
\end{tabular}} & $\stackrel{W}{\to}$ & $\bullet$ & $\stackrel[E]{*1}{\to}$ & $\left[\begin{smallmatrix}1 & x & 1\\
Y-1 & 1 & Y\\
1 & x & 1
\end{smallmatrix}\right]+*0$ & $\shortrightarrow$ & \multicolumn{1}{c|}{(HR)} &  &  &  & \tabularnewline
 & \vphantom{%
\begin{tabular}{c}
\tabularnewline
\tabularnewline
\tabularnewline
\tabularnewline
\end{tabular}} &  &  &  & $\left[\begin{smallmatrix}1 & x & 1\\
0 & 1 & 0\\
1 & x & 1
\end{smallmatrix}\right]+*1$ & $\shortrightarrow$ & \multicolumn{1}{c|}{(MA)} &  &  &  & \tabularnewline
\cline{6-8} \cline{7-8} \cline{8-8} 
 &  &  &  & \multicolumn{1}{c}{} &  &  &  &  &  &  & \tabularnewline
\cline{6-12} \cline{7-12} \cline{8-12} \cline{9-12} \cline{10-12} \cline{11-12} \cline{12-12} 
 & \vphantom{%
\begin{tabular}{c}
\tabularnewline
\tabularnewline
\tabularnewline
\tabularnewline
\end{tabular}} & $\stackrel{NW}{\to}$ & $\bullet$ & $\stackrel[SE]{*1}{\to}$ & $\left[\begin{smallmatrix}0 & X & 1\\
Y & 1 & Y\\
1 & X & 1
\end{smallmatrix}\right]+*0$ & $\stackrel[SW]{W}{\to}$ & $\bullet$ & $\stackrel[W]{SW}{\to}$ & $\left[\begin{smallmatrix}0 & X & 1\\
Y-1 & 1 & Y\\
0 & X & 1
\end{smallmatrix}\right]$ & $\shortrightarrow$ & \multicolumn{1}{c|}{(HR)}\tabularnewline
 & \vphantom{%
\begin{tabular}{c}
\tabularnewline
\tabularnewline
\tabularnewline
\tabularnewline
\end{tabular}} &  &  &  &  & $\stackrel{E}{\to}$ & $\bullet$ & $\stackrel[\,]{SW}{\to}$ & $\left[\begin{smallmatrix}0 & X & 1\\
Y & 1 & Y-1\\
0 & X & 1
\end{smallmatrix}\right]$ & $\shortrightarrow$ & \multicolumn{1}{c|}{(HR)}\tabularnewline
 &  &  &  &  &  & $\stackrel[SE]{C}{\to}$ & $\bullet$ & $\stackrel[C]{SE}{\to}$ & $\left[\begin{smallmatrix}0 & X & 1\\
Y & 0 & Y\\
1 & X & 0
\end{smallmatrix}\right]$ & $\shortrightarrow$ & \multicolumn{1}{c|}{(CR)}\tabularnewline
 & \vphantom{%
\begin{tabular}{c}
\tabularnewline
\tabularnewline
\tabularnewline
\tabularnewline
\end{tabular}} &  &  &  & $\left[\begin{smallmatrix}0 & x & 1\\
1 & 1 & 1\\
1 & x & 0
\end{smallmatrix}\right]+*1$ & $\shortrightarrow$ & (MB) &  &  &  & \multicolumn{1}{c|}{}\tabularnewline
\cline{6-12} \cline{7-12} \cline{8-12} \cline{9-12} \cline{10-12} \cline{11-12} \cline{12-12} 
\end{tabular}}

\caption{\label{fig:TERodd}A winning strategy for the second player in $\protect\TER(M)+*1$.
We assume $x,y\ge1$, $X,Y\ge3$, and $x,y,X,Y$ are odd.}
\end{figure}

We now have enough tools to present our main result about grid games.

\begin{prop}
\label{prop:TER}
For each grid graph, $\nim(\TER(P_{m}\Box P_{n}))=\pty(mn)$.
\end{prop}

\begin{proof}
Corollary~\ref{cor:alpha} implies that $\nim(\TER(P_{m}\Box P_{n}))=\TER(M)$,
where 
\[
M=\left[\begin{smallmatrix}1 & b & 1\\
d & e & d\\
1 & b & 1
\end{smallmatrix}\right]
\]
is the corresponding matrix. We can replace $e$ by $\pty(e)$ by
Proposition~\ref{prop:TERmiddle}. 

If $\pty(mn)=0$, then $b$ or $d$ is even and $e=0$, so the result
follows from Proposition~\ref{prop:TERmatrixEven}. If $\pty(mn)=1$,
then $b$ and $d$ are odd and $e=1$, so the result follows from
Proposition~\ref{prop:TERmatrixOdd}.
\end{proof}

\section{Lattice games}

Let $\Gamma=P_{n_{1}}\Box\cdots\Box P_{n_{d}}$ be a lattice graph. Our option preserving function $\alpha$ maps game positions to $d$-dimensional tensors in $\mathbb{W}^{3\times\cdots\times 3}$. The matrix games generalize to tensor games. In this generalization, the set $S$ contains face sums instead of edge sums.

\begin{prop}
For each lattice graph, $\nim(\DNT(\Gamma))=\pty(n_{1}\cdots n_{d})$.
\end{prop}

\begin{proof}
The central reflection strategy used in Proposition~\ref{prop:DNT}
works for tensor games.
\end{proof}

Just like in Proposition~~\ref{prop:TERmiddle}, the tensor game $\TER(M)$ is isomorphic to $\TER(N){\delay}\mathsf{D}_{e}$, where $N$ is built from $M$ by replacing the central entry $e$ by $0$.  

\begin{prop}
If $\pty(n_{1}\cdots n_{d})=0$, then $\nim(\TER(\Gamma))=0$.
\end{prop}

\begin{proof}
The central entry $e$ of the corresponding tensor $M$ is even, so $\nim(\TER(\Gamma))=\nim(\TER(N))=0$ since the second player can win $\TER(N)$ using the central reflection strategy described in Remark~\ref{rem:TEReven}.
\end{proof}

\begin{prop}
If $\pty(n_{1}\cdots n_{d})=1$, then $\nim(\TER(\Gamma))\ne0$.
\end{prop}

\begin{proof}
The central entry $e$ of the corresponding tensor $M$ is odd, so $\nim(\TER(\Gamma))=\nim(\TER(N){\delay}\mathsf{D}_{1})\ne 0$ since the second player can win $\TER(N)$ by the central reflection strategy.
\end{proof}

It seems reasonable to expect that $\nim(\TER(\Gamma))=1$ for $\pty(n_{1}\cdots n_{d})=1$.
Unfortunately, even the analysis for $\Gamma=P_{3}\Box P_{3}\Box P_{3}$
seems quite complex.

\section*{Acknowledgements}

This material is based upon work supported by the National Science Foundation under Grant No.~DMS-1929284 while the authors were in residence at the Institute for Computational and Experimental Research in Mathematics in Providence, RI, via the Collaborate@ICERM program.

\bibliographystyle{plain}
\bibliography{game}

\end{document}